\documentclass[12pt,twoside]{amsart}
\usepackage{amssymb,amsmath,amsthm}
\usepackage{graphicx}
\usepackage{enumerate}

\newcommand{\norm}[1]{\left\lVert #1 \right\rVert}
\def\A{\mathcal A}
\def\B{\mathcal B}
\def\F{\mathcal F}
\def\overF{\overline{\mathcal F}}
\def\G{\mathcal G}
\def\I{\mathcal I}
\def\K{\mathcal K}
\def\mL{\mathcal L}
\def\N{\mathcal N}
\def\Ps{\mathcal P}
\def\W{\mathcal W}
\def\dps{\displaystyle}
\def\ep{\varepsilon}
\newtheorem{theorem}{Theorem}[section]
\newtheorem{lemma}[theorem]{Lemma}
\newtheorem{definition}[theorem]{Definition}
\newtheorem{corollary}[theorem]{Corollary}
\newtheorem{proposition}[theorem]{Proposition}
\newtheorem{remark}[theorem]{Remark}

\begin{document}

\title[Weaker relatives of the BAP for a Banach operator ideal]{Weaker relatives of the bounded approximation property for a Banach operator ideal}

\author{Silvia Lassalle, Eve Oja and Pablo Turco}

\address{Departamento de Matem\'{a}tica, Universidad de San Andr\'{e}s,
Vito Dumas 284, (B1644BID) Victoria, Buenos Aires, Argentina, FCEN - UBA and IMAS - CONICET.}
\email{slassalle@udesa.edu.ar}

\address{Faculty of Mathematics and Computer Science, University of Tartu,
J. Liivi 2, 50409 Tartu, Estonia; Estonian Academy of Sciences, Kohtu 6, 10130
Tallinn, Estonia}
\email{eve.oja@ut.ee}

\address{IMAS - CONICET, Pab I, Facultad de Cs. Exactas y Naturales, Universidad de Buenos
Aires, (1428) Buenos Aires, Argentina.}
\email{paturco@dm.uba.ar}

\begin{abstract} Fixed a Banach operator ideal $\A$, we introduce and investigate two new approximation properties, which are strictly weaker than the bounded approximation property (BAP) for $\A$ of Lima, Lima and Oja (2010). We call them the weak BAP for $\A$ and the local BAP for $\A$, showing that the latter is in turn strictly weaker than the former. Under this framework, we address the question of approximation properties passing from dual spaces to underlying spaces. We relate the weak and local BAPs for $\A$ with approximation properties
given by tensor norms and show that the Saphar BAP of order $p$ is the weak BAP for the ideal of absolutely $p^*$-summing operators, $1\leq p\leq\infty$, $1/p + 1/{p^*}=1$.
\end{abstract}

\keywords{bounded approximation properties, Banach operator ideals}
\subjclass[2010]{46B28, 46B20,  47L05, 47L20}

\maketitle

\section{Introduction}

Let $X$ be a Banach space and let $1\le \lambda <\infty$. We denote by $\A=(\A, \norm{\cdot}_\A)$ a Banach operator ideal. As usual, $\mathcal L, \mathcal F, \overline{\mathcal F}$, $\K$ and $\W$ are the ideals of bounded, finite rank, approximable, compact and weakly compact linear operators, respectively; all considered with the supremum norm $\|\cdot\|$.

Recall that $X$ has the \textit{approximation property} (AP for short) if its identity map $I_X$ can be uniformly approximated by finite rank operators on compact sets, i.e., there exists a net $(S_\alpha)$ in $\F(X):=\F(X;X)$ such that $S_\alpha \to I_X$ uniformly on compact subsets of $X$. If the net $(S_\alpha)$ can be chosen to satisfy also that $\sup_\alpha \|S_\alpha\| \le \lambda$, then $X$ is said to have the $\lambda$-\textit{bounded approximation property} ($\lambda$-BAP). The 1-BAP is called the \textit{metric approximation property} (MAP). If $X$ has the $\lambda$-BAP for some $\lambda$, then $X$ is said to have the \textit{bounded approximation property} (BAP).

In \cite{LO}, Lima and Oja defined the weak BAP and used it, among others, to approach the famous problem: {\it are the {\rm AP} and the {\rm MAP} equivalent on a dual space?}

Recall that $X$ has the \textit{weak $\lambda$-bounded approximation property} (weak $\lambda$-BAP) if for every Banach space $Y$ and for each operator $T$ in $\W(X;Y)$, there exists a net $(S_\alpha)$ in $\F(X)$ such that $S_\alpha \to I_X$ uniformly on compact subsets of $X$ and $\limsup_\alpha \|TS_\alpha\| \le \lambda\|T\|$. In \cite{LLO1}, Lima, Lima and Oja, continuing to approach the above-mentioned problem, extended the weak BAP as follows.

\begin{definition}[Lima--Lima--Oja]
{\rm A Banach space $X$ has the} $\lambda$-bounded approximation property for $\A$ {\rm ($\lambda$-BAP for $\A$) if for every Banach space $Y$ and for each operator $T$ in $\A(X;Y)$, there exists a net $(S_\alpha)$ in $\F(X)$ such that $S_\alpha \to I_X$ uniformly on compact subsets of $X$ and $\limsup_\alpha \|TS_\alpha\|_\A \le \lambda\|T\|_\A.$}
\end{definition}

The BAP for $\A$ allows the understanding of several known approximation properties in terms of Banach operator ideals and their geometry. For instance, the $\lambda$-BAP is clearly the $\lambda$-BAP for $\mathcal L$, and it is also the $\lambda$-BAP for the ideal $\mathcal{I}$ of integral operators \cite[Theorem~2.1]{LLO1}. The weak $\lambda$-BAP is by definition the $\lambda$-BAP for $\W$, and it is also the $\lambda$-BAP for $\K$ \cite[Theorem~2.4]{LO} and for the ideal $\N$ of nuclear operators \cite[Theorem~3.1]{LLO1}.

From \cite{LO} and \cite{LLO1} it is clear that in the special cases of $\A$ mentioned above, the $\lambda$-BAP for $\A$ is equivalent to its (at least formal) weakening, where the uniform convergence $S_\alpha \to I_X$ on compact subsets of $X$ is replaced by the pointwise convergence. In turn, a weakening of this weakening was occasionally also considered in \cite{LLO1}. Namely, in~\cite[Problem~5.5]{LLO1}, the authors wondered if given an arbitrary Banach operator ideal $\A$, the $\lambda$-BAP for $\A$ could be equivalent to a seemingly weaker property where no ``global'' behavior for the approximating net is required. We shall call this property the local $\lambda$-BAP for $\A$ (see Definition~\ref{def:local for A} below). Problem~5.5 of \cite{LLO1} (see also \cite[Problem 4.1]{Oja_CIDAMA}) has an obvious positive answer if $\A=\mL$. The answer is also positive if $\A =\mathcal W$ or $\A=\K$ \cite[Theorem~3.6]{Oja_strong}.

One of our main aims in the present paper is to show that these two weakenings are not formal (see Sections~2 and~4 below). So, it makes sense to introduce the following concepts.

\begin{definition}
{\rm A Banach space $X$ has the} weak $\lambda$-bounded approximation property for $\A$ {\rm (weak $\lambda$-BAP for $\A$ ) if for every Banach space $Y$ and for each operator $T$ in $\A(X;Y)$, there exists a net $(S_\alpha)$ in $\F(X)$ such that $S_\alpha \to I_X$ pointwise and $\limsup_\alpha \|TS_\alpha\|_\A \le \lambda\|T\|_\A$. }
\end{definition}

\begin{definition}\label{def:local for A}
{\rm A Banach space $X$ has the} local $\lambda$-bounded approximation property for $\A$ {\rm (local $\lambda$-BAP for $\A$) if for every Banach space $Y$ and for each operator $T$ in $\A(X;Y)$, there exists a net $(T_\alpha) $ in $\mathcal F(X;Y)$ such that $T_\alpha \to T$ pointwise and $\limsup_\alpha \|T_\alpha\|_\A \le \lambda \|T\|_\A$.}
\end{definition}

Remark that the local $\lambda$-BAP for $\K$ was considered in \cite{Oja_strong} under the name of condition $c^*_\lambda$. It is interesting and also important to note that the local $\lambda$-BAP for the ideal $\Ps_p$ of absolutely $p$-summing operators was considered, implicitly, without giving any name, already in 1972 by Saphar \cite{Sap}. Namely, in~\cite[Theorem~2]{Sap}, Saphar characterized his $\lambda$-BAP of order $p$ (this is, by definition, the $\lambda$-BAP which is given by the Chevet--Saphar tensor norm $g_p$; see Section~\ref{Sec: tensor norms}) as follows. For $1\le p \le \infty$ we denote by $p^*$ the conjugate index of $p$, i.e., $1/p + 1/{p^*}=1$ with the usual convention that $p^*=1$ if $p=\infty$.

\begin{theorem}[Saphar] \label{Saphar_orig}
Let $1\le \lambda < \infty$ and $1\le p \le \infty$. A Banach space $X$ has the $\lambda$-BAP of order $p^*$ if and only if $X$ has the local $\lambda$-BAP for $\Ps_p$.
\end{theorem}

Summarizing we have:
\begin{equation*}\label{implicaciones}
\lambda\text{-BAP} \ \Rightarrow \
\lambda\text{-BAP for } \A \Rightarrow
\text{ weak } \lambda\text{-BAP for } \A \Rightarrow \text{ local } \lambda\text{-BAP for } \A.
\end{equation*}
We do not know of any example of an ideal $\A$ for which the $\lambda$-BAP is strictly stronger than the $\lambda$-BAP for $\A$. However, as already mentioned, we shall show that the subtle differences between the $\lambda$-BAP for $\A$, the weak $\lambda$-BAP for $\A$ and the local $\lambda$-BAP for $\A$ are, in fact, not formal (see Sections~2 and~4 for examples).

The paper is organized as follows. In Section~\ref{Sec: BAP's for A} we study the weak BAP for $\A$, the local BAP for $\A$ and the interplay between them. We exhibit classes of ideals for which they coincide (Theorem~\ref{wBAP coincide localBAP}) and also examples for which they differ (Proposition~\ref{weak BAP not local}). Also, we give an omnibus characterization of the weak BAP for $\A$ (Theorem~\ref{Omnibus Thm}) which allows us to relate this property and the BAP for $\A$. In Section~\ref{Sec: lifting} we relate these approximation properties with some other approximation properties, also determined by Banach operator ideals, showing that they pass from a dual space down to the underlying space, giving there the corresponding metric approximation properties. In order to do so, we show that (for many operator ideals) it is enough to check the weak and the local BAPs for $\A$ using only bidual spaces. Finally, in Section~\ref{Sec: tensor norms} we connect the weak and the local BAPs for $\A$ with approximation properties given by tensor norms (Theorem~\ref{Teo Saphar Gral}) extending, among others, Theorem~\ref{Saphar_orig} of Saphar (Corollary~\ref{Saphar again}). As a by-product, we show that every Banach space has the local MAP for the ideal of $p$-integral operators $\I_p$, $1\le p\le \infty$ (Corollary~\ref{c^* I_p}), and that this property may differ from the weak BAP for $\I_p$, $2< p < \infty$ (Proposition~\ref{c* and weak BAP Ip}).
\medskip

All the relevant terminology and preliminaries will be given in corresponding sections. For the theory of operator ideals we refer the reader to the books of Pietsch~\cite{Pie}, of Defant and Floret \cite{DF}, of Diestel, Jarchow and Tonge~\cite{djt} and of Ryan~\cite{Ryan}. For approximation properties we refer the reader to the books of Lindenstrauss and Tzafriri~\cite{LT}, of Diestel, Fourie and Swart \cite{dfs} and to the books \cite{DF, Ryan}; see also the surveys \cite{Cas, Oja_Survey} and references therein.

Our notation is standard. We consider Banach spaces $X$, $Y$ over the same, either real or complex, field $\mathbb K$. We denote by $X^{*}$ and $B_X$ the topological dual of $X$ and its closed unit ball, respectively. The canonical inclusion of $X$ into its bidual $X^{**}$ is denoted by $J_X$. The Banach space of all absolutely $p$-summable sequences in $X$ is denoted by $\ell_p(X)$ and its norm by $\norm{\cdot}_p$, for any $1\le p<\infty$, and the Banach space of all null sequences in $X$ is denoted by $c_0(X)$, considered with the supremum norm. As usual, operators in $\mathcal F(X;Y)$ are regarded as elements of the algebraic tensor product $X^{*}\otimes Y$ and tensors in $X\otimes Y$ as operators in $\mathcal F (X^{*};Y)$. Also, $\tau_w, \tau_s$ and $\tau_c$ stand for the weak operator topology, the strong operator topology and the compact open topology, respectively; all considered on $\mL(X;Y)$.

\section{Three bounded approximation properties for $\A$}
\label{Sec: BAP's for A}

Let us start with a couple of preliminary observations showing, among others, that the problem~\cite[Problem~5.5]{LLO1} mentioned in the Introduction has a negative answer. Although some counterexamples had been at hand in several articles, they were not explicitly written. For instance, $\N$ immediately provides a counterexample, due to the reasons given below. Also, the next result can be deduced from~\cite{Oja_JMAA}. We shall use that $\overline{\F(X;Y)}^{\norm{\cdot}_\A}=\A(X;Y)$ for all Banach spaces $X$ and $Y$ whenever $\A$ is a minimal Banach operator ideal.

\begin{proposition}\label{No Prop 5.5}
Every Banach space $X$ has the local MAP for $\A$ whenever $\A$ is a minimal Banach operator ideal. As a consequence, the local MAP for $\A$ does not imply the BAP for $\A$ whenever $X$ fails the AP and $\A$ is a minimal Banach operator ideal.
\end{proposition}

\begin{proof}
Let $\A$ be a minimal Banach operator ideal and $X, Y$ be Banach spaces. Given
$T \in \A(X;Y)$ there exists a sequence $(T_n) \subset \mathcal F(X;Y)$ such
that $T_n \rightarrow T$ in $\A$ (and therefore $T_n \rightarrow T$ pointwise).
Then, $\lim_n \|T_n\|_\A = \|T\|_\A$, showing that $X$ has the local MAP for
$\A$. In particular, this is true for any Banach space $X$ without the AP and
hence without the BAP for $\A$.
\end{proof}

Thus, \textit{the local MAP for $\A$}, in general, \textit{does not imply the
AP}. On the other hand, for instance, \textit{the AP does not imply the local
BAP for $\Ps_p$ for any $p\ne 2$}. This follows from Theorem~\ref{Saphar_orig}
and the fact, due to Reinov~\cite[Corollary~3.1]{Rei}, that there is a Banach
space with the AP which lacks the approximation property of order $q$ for any
$q\ne2$.

Well-known examples of minimal Banach operator ideals include $\overF$ and
$\N$. As we see next, if $\A$ equals one of these, then the local BAP for $\A$
is strictly weaker than the weak BAP for $\A$.

\begin{proposition}\label{weak BAP not local}
Every Banach space $X$ has the local MAP for $\overline \F$ and $\N$. If
$X$ fails the AP, then $X$ does not have the weak BAP for $\overline \F$ nor $\N$.
\end{proposition}

\begin{proof}
In the both cases $\A=\overline \F$ and $\A=\N$, the weak $\lambda$-BAP for $\A$ is the same as the $\lambda$-BAP for $\A$. For $\mathcal N$, this was proved in~\cite{LLO1} (see the proof of Theorem~3.1 in \cite{LLO1} or~\cite[Theorem~1.2]{LLO2}). For $\overF$, see Proposition~\ref{BAP for F} below.
\end{proof}

\begin{proposition} \label{BAP for F}
Let $X$ be a Banach space and $1\le \lambda < \infty$. Then, the following
statements are equivalent.
\begin{enumerate}[\upshape (i)]
 \item $X$ has the weak $\lambda$-BAP.
 \item $X$ has the $\lambda$-BAP for $\overF$.
 \item $X$ has the weak $\lambda$-BAP for $\overF$.
\end{enumerate}
\end{proposition}

\begin{proof}
The implications (i) $\Rightarrow$ (ii) $\Rightarrow$ (iii) are clear. As the weak $\lambda$-BAP and the weak $\lambda$-BAP for $\K$ coincide \cite[Theorem~2.4]{LO}, to complete the proof we
show that the latter property is implied by (iii). Fix a Banach space $Y$ and $T\in \K(X;Y)$. Denote by $j\colon Y\to C(K)$ a linear isometric embedding for a suitable compact space $K$. Since $C(K)$ has the AP, $\K(X;C(K))=\overF(X;C(K))$. Hence, $jT\in \overF(X;C(K))$. By (iii), there exists a net $(S_\alpha)$ in $\F(X)$ such that $S_\alpha \to I_X$ pointwise and $\limsup_\alpha \|jTS_\alpha\| \le \lambda\|jT\|$. Being $j$ an isometry, the result follows.
\end{proof}

The operator ideals which are Banach with respect to the usual norm $\norm{\cdot}$ are called {\it closed} (see \cite{Pie}) or {\it classical} (see \cite{djt}). A wide list of closed operator ideals can be found in~\cite{GoGu00}, for instance. The inclusion $\A\subset \mathcal B$ (defined as $\A(X;Y)\subset \mathcal B(X;Y)$ for all Banach spaces $X$ and $Y$) provides a
natural partial ordering on the family of all operator ideals. In the family of all closed operator ideals, $\mL$ is the largest element and $\overF$ is the smallest one.

Proposition~\ref{BAP for F} shows that the smallest closed ideal $\overF$
\textit{yields the weak} BAP, meaning that the weak $\lambda$-BAP and the
$\lambda$-BAP for $\overF$ coincide. It is not known (see
\cite[Problem~5.3]{LLO1}) whether there is the largest closed operator ideal
yielding the weak BAP. (Note that $\mL$ trivially yields the BAP.) To our
knowledge, the best result belongs to Lissitsin \cite{Liss}: the weak
$\lambda$-BAP is equivalent to the $\lambda$-BAP for $\mathcal{RN}^{dual}$,
the ideal of operators whose adjoints are Radon--Nikod\'ym.

We saw that, in the case of the closed operator ideals $\A=\overF, \K, \W$ and
$\mL$, the BAP for $\A$ coincides with the weak BAP for $\A$. We shall show that
this is true for \textit{any} closed operator ideal $\A$ (see Corollary~\ref{BAP
for A = weak BAP for A}). Hence, in particular, for $\A=\mathcal{RN}^{dual}$.

Our next goal is to establish an omnibus characterization of the weak BAP for
$\A$ (Theorem~\ref{Omnibus Thm}). This is one of our main results which will be
used throughout the paper. In order to proceed, recall that Grothendieck's
characterization (see for example \cite[Proposition~1.e.3]{LT}) states that, algebraically,
\begin{equation}\label{Groth charact}
(\mathcal L(X;Y), \tau_c)^*=Y^*\widehat{\otimes}_\pi X,
\end{equation}
the projective tensor product, under the duality
$$
\langle u, T \rangle= \sum_{n=1}^{\infty} y^*_n(T x_n),\quad
u=\sum_{n=1}^{\infty} y^*_n\otimes x_n \in Y^*\widehat{\otimes}_\pi X,\quad
T\in \mL(X;Y).
$$

Recall also that
$$
\begin{array}{rl}
Y^*\widehat{\otimes}_\pi X & =\Big\{\dps u=\sum_{n=1}^{\infty} y^*_n \otimes
x_n\colon \ (y^*_n) \in \ell_1(Y^*), \ (x_n) \in c_{0}(X)\Big\} \\
& =\Big\{\dps u =\sum_{n=1}^{\infty} y^*_n \otimes x_n \colon
(y^*_n) \in c_{0}(Y^*), \ (x_n) \in \ell_1(X)\Big\}.\end{array}
$$
Here it will be convenient to replace the null sequences with the $\A$-null
sequences of Carl and Stephani \cite{CaSt}, defined as follows.

Fixed an operator ideal $\A$, a sequence $(x_n)$ in a Banach space $X$ is said
to be \textit{$\A$-null} if there exist a Banach space $Z$, an operator $R\in
\A(Z;X)$ and a null sequence $(z_n) \subset Z$ such that $x_n=Rz_n$ for all$n
\in\mathbb N$ (see \cite[Definition~1.1 and Lemma~1.2]{CaSt}). The set
$c_{0,\A}(X)$ of the $\A$-null sequences in $X$ forms a linear subspace of
$c_0(X)$. Now we consider the following linear subspaces of
$Y^* \widehat{\otimes}_\pi X$:
$$
\begin{array}{rl}
\G_\A & :=\Big\{ \dps u=\sum_{n=1}^{\infty} y^*_n \otimes x_n\colon \ (y^*_n) \in
\ell_1(Y^*), \ (x_n) \in c_{0,\A}(X)\Big\}, \\
\G^\A & :=\Big\{\dps u =\sum_{n=1}^{\infty} y^*_n \otimes x_n \colon
(y^*_n) \in c_{0,\A}(Y^*), \ (x_n) \in \ell_1(X)\Big\}.\end{array}
$$
Associated to these subspaces we have natural locally convex Hausdorff
topologies $\tau_\A :=\sigma(\mL (X;Y), \G_\A)$ and $\tau^\A:= \sigma(\mL
(X;Y),\G^\A)$. Then, as is well known, we may identify $(\mL (X;Y), \tau_\A)^*=
\G_\A$ and $(\mL (X;Y), \tau^\A)^*= \G^\A$, similarly to \eqref{Groth charact}.

Since $\tau_w=\sigma(\mathcal L(X;Y), Y^* \otimes X)$, we clearly have $\tau_w
\subset \tau_\A, \tau^\A \subset \tau_c$. By \cite[Proposition~1.4 and
Remark~1.3]{LaTur2}, $c_0(X)=c_{0,\overF}(X)$, and therefore
$\tau_{\overF}=\tau^{\overF}=\tau_c$. Hence, we have the following.

\begin{proposition} \label{tau A = tau c}
Let $\A$ be an operator ideal. If $\overF\subset \A$, then
$\tau_\A=\tau^\A=\tau_c$.
\end{proposition}

We shall need a natural modification of the BAP for $\A$.

\begin{definition}\label{def: BAP for A and tau}
{\rm Let $\A$ be a Banach operator ideal and $1\le \lambda < \infty$. Let $X$ be
a Banach space and $\tau$ a topology on $\mathcal L(X)$. We say that $X$ has
the} $\lambda$-bounded approximation property for $\A$ and $\tau$ {\rm
($\lambda$-BAP for $\A$ and $\tau$) if for every Banach space $Y$ and for each
operator $T$ in $\A(X;Y)$, there exists a net $(S_\alpha)$ in $\F(X)$ such that
$S_\alpha \to I_X$ in $\tau$ and $\limsup_\alpha \|TS_\alpha\|_\A \le
\lambda\|T\|_\A$.}
\end{definition}

Clearly, the $\lambda$-BAP for $\A$ is precisely the $\lambda$-BAP for $\A$
and $\tau_c$. And the weak $\lambda$-BAP for $\A$ is the $\lambda$-BAP for $\A$
and $\tau_s$. It also coincides with the $\lambda$-BAP for $\A$ and $\tau_w$
because $\tau_w$ and $\tau_s$ are the same on convex sets (see, for instance,
\cite[Corollary~VI.1.5]{DS}).

Now, we are in conditions to state and prove the omnibus characterization of
the weak BAP for $\A$ which can be seen as a generalization of \cite[Theorem
2.4]{LO} (from $\W=(\W, \norm{.})$ to $\A=(\A, \norm{.}_\A)$); see also
Remark~\ref{DFJP-LNO} concerning methods of proof.

\begin{theorem}\label{Omnibus Thm}
Let $\A$ be a Banach operator ideal and $1\le \lambda < \infty$. For a Banach
space $X$, the following statements are equivalent.
\begin{enumerate}[\upshape (i)]
\item $X$ has the weak $\lambda$-BAP for $\A$.

\item For every Banach space $Y$ and for each operator $T \in \A(X;Y)$,
there exists a net $(S_{\alpha})$ in $\mathcal F(X)$ with
$\limsup_{\alpha}\|TS_{\alpha}\|_\A \leq \lambda \|T\|_\A$ such that
$TS_{\alpha}\rightarrow T$ pointwise.

\item For every Banach space $Y$ and for each operator $T \in \A(X;Y)$ with
$\|T\|_\A=1$, for all sequences $(y^*_n)$ in $Y^*$ and $(x_n)$ in $X$ such that
$\dps\sum_{n=1}^{\infty} \|y^*_n\|\|x_n\|<\infty$, one has the inequality
$$
|\sum_{n=1}^{\infty} y^*_n(Tx_n)|\leq \lambda \sup\limits_{\substack{ \|TS\|_\A
\leq 1 \\ S \in \mathcal F(X)}}|\sum_{n=1}^{\infty} y^*_n(TSx_n)|.
$$
\item $X$ has the $\lambda$-BAP for $\A$ and $\tau^{\A^{dual}}$.
\end{enumerate}
\end{theorem}

\begin{proof}
Clearly, (i) implies (ii). Also, (iv) implies (i) since the weak $\lambda$-BAP
for $\A$ coincides with the weak $\lambda$-BAP for $\A$ and $\tau_w$. To prove
that (ii) implies (iii), follow the easy straightforward proof of~\cite[Theorem
2.4, (a) $\Rightarrow$ (d)]{LO} with the obvious modifications.

Let us prove that (iii) implies (iv). Fix a Banach space $Y$ and $T\in
\A(X;Y)$ such that $\|T\|_\A =1$. Consider the absolutely convex set
$$
M=\{S \in \mathcal F(X) \colon \|TS\|_\A \leq \lambda\},
$$
and suppose that $I_X \notin \overline{M}\;^{\tau^{\A^{dual}}}$. Then, there
exists $\phi \in (\mL(X);\tau^{\A^{dual}})^*$ such that
$$
|\phi(I_X)|> \sup\{|\phi(S)| \colon S \in M\}.
$$
We may write $\phi=\sum_{n=1}^{\infty} x^*_n\otimes x_n$ with
$(x^*_n)\in c_{0,\A^{dual}}(X^{*})$ and $(x_n) \in \ell_1(X)$. Hence,
\begin{equation}\label{eq1sot}
\dps |\sum_{n=1}^{\infty} x^*_n(x_n)|> \sup\limits_{\substack{\|TS\|_\A \leq
\lambda\\S \in \mathcal F(X) }} |\sum_{n=1}^{\infty} x^*_n(Sx_n)|=\lambda
\sup\limits_{\substack{ \|TS\|_\A \leq 1\\S \in \mathcal F(X)}}
|\sum_{n=1}^{\infty} x^*_n(Sx_n)|.
\end{equation}

We affirm that inequality \eqref{eq1sot} cannot hold. Indeed, since the
sequence $(x^*_n)$ is $\A^{dual}$-null, given $\ep>0$ there exist a Banach space
$Z$, an operator $R\in \A^{dual}(Z;X^{*})$, meaning that $R^*\in
\A(X^{**};Z^*)$, and a null sequence $(z_n)$ in $B_Z$ such that $x^*_n=Rz_n$ for
all $n$. Then $R^*J_X\in \A(X;Z^*)$. Consider the Banach space $W=Y \times Z^*$
endowed with the sum norm. Fix $r> 0$ and define the operator $\widetilde{T}
\colon X \rightarrow W$ by $\widetilde{T}x= (Tx, rR^*J_Xx)$, $x\in X$. Then
$\widetilde{T} \in \A(X;W)$
and
$$
\|\widetilde{T}\|_{\A} \leq 1 + r\|R^*J_X\|_\A.
$$
As an element of $W^*$, $(0,z_n)$ satisfies
$$
(0,z_n) (\widetilde Tx)=r (R^*J_Xx)(z_n)=r (Rz_n)(x)=rx_n^*(x),
$$
for all $x \in X$ and all $n$. Then, by (iii), we have
\begin{equation}\label{eq2sot}
\begin{array}{rl}
\dps |\sum_{n=1}^\infty rx^*_n(x_n)|&=\dps|\sum_{n=1}^\infty (0,z_n)(\widetilde
Tx_n)| \\
& \leq \dps \lambda \sup\limits_{\substack{\|\widetilde TS\|_\A \leq
1+r\|R^*J_X\|_\A\\S \in \mathcal F(X)}} |\sum_{n=1}^\infty (0,z_n)(\widetilde T
Sx_n)|\\
& \dps =\lambda(1+r\|R^*J_X\|_\A) \sup\limits_{\substack{ \|\widetilde TS\|_\A
\leq 1\\S \in \mathcal F(X)}} |\sum_{n=1}^\infty r x^*_n(Sx_n)|.
\end{array}
\end{equation}
Since $T=P_Y\widetilde T$, where $P_Y$ is the norm one projection of $W$ onto
$Y$, we have $\|TS\|_\A\leq \|\widetilde T S\|_\A$, for any $S \in \mathcal
F(X)$. Hence from \eqref{eq2sot}, we obtain
$$
|\sum_{n=1}^\infty x^*_n(x_n)|\leq \lambda (1+r\|R^*J_X\|_\A)
\sup\limits_{\substack{\|TS\|_\A \leq 1\\S \in \mathcal F(X)}}
|\sum_{n=1}^\infty x^*_n(Sx_n)|.
$$
Since $r>0$ is arbitrary, we conclude that
$$
|\sum_{n=1}^\infty x^*_n(x_n)|\leq \lambda \sup\limits_{\substack{ \|TS\|_\A
\leq 1\\S \in \mathcal F(X) }} |\sum_{n=1}^\infty x^*_n(Sx_n)|,
$$
contradicting inequality \eqref{eq1sot}. Therefore, the proof is complete.
\end{proof}

\begin{remark}\label{DFJP-LNO}
{\rm Up to inequality~\eqref{eq1sot}, our proof of the implication (iii)
$\Rightarrow$ (iv) followed the beginning of the proof of~\cite[Theorem~2.4,
(d$'$) $\Rightarrow$ (a$'$)]{LO}. However, the main part of our proof
essentially differs from that in~\cite{LO}. Namely,~\cite{LO} relied on the
isometric version of the Davis, Figiel, Johnson and Pe{\l}czy\'nski
factorization lemma due to Lima, Nygaard and Oja~\cite{LNO}. Our proof cannot
use this factorization result because its suitable version seems to be unknown
for arbitrary Banach operator ideals. So, since in the case when $\A=\W$, one
has $\A=\A^{dual}$ \cite[Proposition~4.4.7]{Pie} and $\tau^{\A^{dual}}=\tau_c$
(see Proposition~\ref{tau A = tau c}), we have given as a by-product an
alternative proof of a main part of~\cite[Theorem~2.4]{LO}.}
\end{remark}

\begin{remark}
{\rm Let $\A$ be an operator ideal. By~\cite[Definition~1.2]{CaSt},
$c_{0,\A}(X)=c_{0,\A^{sur}}(X)$ for any Banach space $X$. Since also $\A^{dual\
sur}=\A^{inj\ dual}$ (see \cite[Theorem~4.7.16]{Pie}), we get that
$\tau^{\A^{dual}}=\tau^{\A^{dual\ sur}}=\tau^{\A^{inj\ dual}}$. Therefore,
condition (iv) of Theorem~\ref{Omnibus Thm} can be stated with $\tau^{\A^{dual\
sur}}$ or with $\tau^{\A^{inj\ dual}}$.}
\end{remark}

As a consequence of Theorem~\ref{Omnibus Thm}, together with the above remark
and Proposition~\ref{tau A = tau c}, we have the following.

\begin{corollary}\label{Coro con F}
Let $\A$ be a Banach operator ideal and $1\le \lambda <\infty$. If $\overF\subset \A^{inj\ dual}$, then a Banach space $X$ has the $\lambda$-BAP for $\A$ if and only if $X$ has the weak
$\lambda$-BAP for $\A$.
\end{corollary}

The above corollary applies to any closed ideal. In this case it can be restated as follows.

\begin{corollary} \label{BAP for A = weak BAP for A}
Let $\A$ be a closed operator ideal and $1\le \lambda <\infty$. Then a Banach
space $X$ has the $\lambda$-BAP for $\A$ if and only if $X$ has the weak
$\lambda$-BAP for $\A$.
\end{corollary}

Corollary~\ref{Coro con F} can also be applied to non-closed operator ideals
such as the ideal of $\infty$-integral operators $\I_\infty$.

\begin{proposition}
Let $1\le \lambda <\infty$. Then a Banach space $X$ has the $\lambda$-BAP for
$\I_\infty$ if and only if $X$ has the weak $\lambda$-BAP for $\I_\infty$.
\end{proposition}

\begin{proof}
Recall that given $Z$ and $Y$ Banach spaces, $S\in \I_\infty(Z;Y)$ if and only
if $J_YS$ factorizes through a $C(K)$-space. Hence $\I_\infty\ne \mL$, which
allows us to observe that $\I_\infty$ is non-closed. Indeed, by its
definition~\cite[19.3.1,~19.3.9]{Pie}, $\I_\infty$ is maximal. But the only
maximal closed ideal is $\mL$ \cite[4.9.7]{Pie}.

Let us also observe that $\I_\infty^{inj}=\mL$. Indeed, by \cite[17.12(4)]{DF},
$\I_\infty^{inj}$ is associated with the Chevet--Saphar tensor norm $g_\infty$.
Therefore (see \cite[Theorem~20.11]{DF}) $\I_\infty^{inj}$ is associated with
$g_\infty\setminus$ which equals the injective tensor norm $\varepsilon$ (see
\cite[Proposition~20.14(5)]{DF}). The claim follows since $\varepsilon$ and $\mL$
are associated (see \cite[17.12(1)]{DF}. Hence $\I_\infty^{inj}=\mL$ and $\overF
\subset \mL =\mL^{dual}=\I_\infty^{inj\ dual}$.
\end{proof}

We saw (Proposition~\ref{weak BAP not local}) that the weak $\lambda$-BAP for
$\A$ and the local $\lambda$-BAP for $\A$ may differ. However, they coincide
for injective Banach operator ideals.

\begin{theorem}\label{wBAP coincide localBAP}
Let $\A$ be an injective Banach operator ideal and $1\le \lambda <\infty$.
Then a Banach space $X$ has the weak $\lambda$-BAP for $\A$ if and only
if $X$ has the local $\lambda$-BAP for $\A$.
\end{theorem}

\begin{proof}
Assume that $X$ has the local $\lambda$-BAP for $\A$. Let us show
that condition (iii) of Theorem~\ref{Omnibus Thm} holds. Let $Y$ be a Banach
space and take $T\in \A(X;Y)$. Set $Z=\overline{T(X)}$ and denote by $T_0$ the
operator $T$ with values in $Z$. Since $\A$ is injective, applying
\cite[Proposition 8.4.4]{Pie}, we know that $T_0\in \A(X;Z)$ and
$\|T_0\|_\A=\|T\|_\A$. By assumption, there exists a net $(T_\alpha) \in
\F(X;Z)$ such that $T_\alpha \rightarrow T_0$ pointwise and
$$
\limsup_{\alpha} \|T_\alpha\|_\A\leq \lambda \|T_0\|_\A=\lambda\|T\|_\A.
$$

Let us order the set of pairs $(\alpha,\ep)$ where $\alpha$ is as above and
$\ep>0$ in a natural way: $(\alpha,\ep)\ge (\tilde \alpha,\tilde\ep)$ if and
only if $\alpha\ge \tilde \alpha$ and $\ep \le \tilde \ep$. For each
$(\alpha,\ep)$ look at the operator $T_\alpha$ which is of the form
$$
T_\alpha=\sum_{j=1}^n x^*_j \otimes z_j\in X^*\otimes Z,
$$
for some $z_1,\ldots,z_n \in Z$ and $x^*_1,\ldots, x^*_n \in X^*$ with
$\sum_{j=1}^{n}\|x^*_j\|=1$. Choose $x_j \in X$ such that $\|Tx_j-z_j\| < \ep$,
$j=1,\dots,n$. Let $S_{(\alpha,\ep)}\in \F(X)$ be the finite rank operator
defined by
$$
S_{(\alpha,\ep)}=\sum_{j=1}^n x^*_j \otimes x_j.
$$
Then
$$
\|TS_{(\alpha,\ep)}-T_\alpha\|_\A=\|\sum_{j=1}^n x^*_j \otimes (Tx_j-z_j)\|_\A
\leq \sum_{j=1}^n\|x^*_j\|\|Tx_j-z_j\|< \ep.
$$
Therefore,
$$
\limsup_{(\alpha,\ep)} \|TS_{(\alpha,\ep)}\|_\A \le\limsup_{\alpha}
\|T_\alpha\|_\A \leq \lambda \|T\|_\A
$$
and for every $x\in X$
$$
\|TS_{(\alpha,\ep)}x-Tx\| \leq \|T_\alpha x-Tx\|+\ep\|x\|,
$$
implying that $TS_{(\alpha,\ep)}\rightarrow T$ pointwise, which completes the
proof.
\end{proof}

\begin{proposition}\label{BAP = local BAP inj closed}
Let $\A$ be an injective closed operator ideal and $1\le \lambda <\infty$.
Then a Banach space $X$ has the $\lambda$-BAP for $\A$ if and only if $X$ has
the local $\lambda$-BAP for $\A$.
\end{proposition}

\begin{proof}
The result follows as a direct application of Corollary~\ref{BAP for A = weak
BAP for A} and Theorem~\ref{wBAP coincide localBAP}.
\end{proof}

Proposition~\ref{BAP = local BAP inj closed} applies, among others, to $\K$,
$\W$, $\mathcal{RN}$, Asplund or $\mathcal{RN}^{dual}$, Rosenthal, Banach--Saks,
completely continuous, weakly completely continuous, unconditionally
converging, separable range, strictly singular and absolutely continuous
operators. The particular case of Proposition~\ref{BAP = local BAP inj closed}
when $\A=\K$ should be compared with \cite[Theorem~3.6]{Oja_strong}.

We shall need the following result which is immediate from Theorem~\ref{wBAP
coincide localBAP}, because $\Ps_p$ is an injective Banach operator ideal
($\Ps_\infty=\mL$ is a trivial case).

\begin{corollary}\label{Coro Saphar}
Let $1\le \lambda < \infty$ and $1\leq p \leq \infty$. Then a Banach space $X$
has the weak $\lambda$-BAP for $\Ps_p$ if and only if $X$ has the local
$\lambda$-BAP for $\Ps_p$.
\end{corollary}

The above corollary nicely completes Saphar's Theorem~\ref{Saphar_orig}; this
will be used in the next two results.

\begin{proposition}
Let $1\le \lambda < \infty$. If a Banach space $X$ has the weak $\lambda$-BAP,
then $X$ has the weak $\lambda$-BAP for $\Ps_p$, $1 < p < \infty$.
\end{proposition}

\begin{proof}
Thanks to \cite[Proposition~4.4]{Oja_TAMS}, $X$ has the Saphar $\lambda$-BAP of
order $p$ whenever $1 < p < \infty$. By Theorem~\ref{Saphar_orig}, $X$ has the
local $\lambda$-BAP for $\Ps_p$ and, by Corollary~\ref{Coro Saphar}, $X$ has the
weak $\lambda$-BAP for $\Ps_p$, $1 < p < \infty$.
\end{proof}

\begin{proposition}\label{ContraejemploPi_p}
There exists a Banach space with the weak MAP for $\Ps_p$ for all $1\leq p\le
2$, which lacks the BAP for $\Ps_p$.
\end{proposition}

\begin{proof}
Let $X$ be a Banach space with cotype 2 and without the AP, which exists by
\cite{Sza}. Then $X$ lacks the BAP for $\A$ for any Banach operator ideal $\A$.
In particular, $X$ lacks the BAP for $\Ps_{p}$. Since $X$ has cotype 2, it has
the Saphar MAP of order $q$ for any $q\geq2$ \cite[p. 126]{Rei} (see also
\cite[pp.~280--281]{DF}). By Theorem~\ref{Saphar_orig}, $X$ has the local MAP
for $\Ps_{p}$ and, by Corollary~\ref{Coro Saphar}, $X$ has the weak MAP for
$\Ps_{p}$ for any $p\le 2$.
\end{proof}

As a consequence of the above and at the light of Proposition~\ref{No Prop
5.5}, the class $\Ps_p$ of $p$-summing operators, $1\leq p\leq2$, provides an
example of other type of ideals (not minimal) which also answers \cite[Problem
5.5]{LLO1} (see the Introduction) by the negative.

\section{Lifting of some approximation properties from $X^*$ to related metric
approximation properties of $X$}
\label{Sec: lifting}

By the well-known Grothendieck's classics, the AP passes down from dual spaces
to underlying spaces. A lifting result due to Lima and Oja asserts that, in this
case, the AP of underlying spaces is always weakly metric (see
\cite[Theorem~2.4]{LO}; for a very simple proof of this result, see \cite[p.
5838, (3)]{Oja_TAMS}).

In this section we shall demonstrate that a similar phenomenon occurs in the
general context of approximation properties determined by Banach operator ideals
$\A$ (see the results from Proposition~\ref{A-AP and wBAP} till
Corollary~\ref{Omnibus Corollary}). Among others, with the particular case of
$\A=\K$ we cover the Lima--Oja result (see text after Proposition~\ref{A-AP and
MAP surjective closed}). To this end, let us show that for many Banach
operator ideals $\A$ it is enough to check the definitions of the $\lambda$-BAP
for $\A$, the weak $\lambda$-BAP for $\A$ and the local $\lambda$-BAP for $\A$
using the bidual spaces instead of all Banach spaces.

\begin{proposition}\label{wBAP for reg}
Let $\A$ be a Banach operator ideal and $1\le \lambda < \infty$. Let $X$ be a
Banach space and $\tau$ a topology on $\mathcal L(X)$. Then the following
statements are equivalent.
\begin{enumerate}[\upshape (i)]
\item $X$ has the $\lambda$-BAP for $\A^{reg}$ and $\tau$.
\item For every Banach space
$Y$ and for each operator
$T \in \A(X;Y^*)$, there exists a net $(S_\alpha)$ in $\F(X)$ such that
$S_\alpha \to I_X$ in $\tau$ and
$$
\limsup_\alpha \|TS_\alpha\|_\A \le \lambda\|T\|_\A.
$$
\item For every Banach space
$Y$ and for each operator
$T \in \A(X;Y^{**})$, there exists a net $(S_\alpha)$ in $\F(X)$ such
that
$S_\alpha \to I_X$ in $\tau$ and
$$
\limsup_\alpha \|TS_\alpha\|_\A \le \lambda\|T\|_\A.
$$
\end{enumerate}
\end{proposition}

\begin{proof}
Note that for every Banach space $Y$, using that $A=(J_Y)^*J_{Y^*}A$ for $A\in
\mL(X;Y^*)$, it is straightforward to verify that $\A(X;Y^*)=\A^{reg}(X;Y^*)$
isometrically. Hence, (i) implies (ii). It is clear that (ii) implies (iii).
Finally, to see that (iii) implies (i) take a Banach space $Y$ and $T\in
\A^{reg}(X;Y)$. Since $J_Y T \in \A(X;Y^{**})$, by assumption, there is a net
$(S_\alpha)$ in $\F(X)$ such that $S_\alpha \rightarrow I_X$ in $\tau$ and
$$
\limsup_{\alpha} \|J_Y T S_\alpha\|_\A \leq \lambda\|J_YT\|_\A,
$$
meaning that
$$
\limsup_{\alpha} \|T S_\alpha\|_{\A^{reg}} \leq \lambda\|T\|_{\A^{reg}},
$$
and the proof is complete.
\end{proof}

Recall that a Banach operator ideal $\A$ is \textit{regular} if $\A^{reg}=\A$.
Note that a lot of Banach operator ideals are regular, such as $\A^{dual},
\A^{max}, \A^{inj}$ for any Banach operator ideal $\A$.

\begin{corollary}\label{regular and biduals}
For a regular Banach operator ideal $\A$, it is enough to check the definition
of the BAP for $\A$ and $\tau$ using bidual spaces, for any topology $\tau$.
\end{corollary}

\begin{proposition}\label{c* for bidual}
Let $\A$ be an injective Banach operator ideal and $1\le\lambda<\infty$. Then a
Banach space $X$ has the local $\lambda$-BAP for $\A$ if and only if for every
Banach space $Y$ and each operator $T \in \A(X;Y^{**})$ there exists a net
$(T_\alpha)$ in $\mathcal F(X;Y^{**})$ such that $T_\alpha \to T$ pointwise and
$$
\limsup_\alpha \|T_\alpha\|_\A \le \lambda \|T\|_\A.
$$
\end{proposition}

\begin{proof}
Let $Y$ be a Banach space and $T\in \A(X;Y)$. Then $J_YT \in \A(X;Y^{**})$ and
there is a net $(S_\alpha) \subset \F(X;Y^{**})$ such that $S_\alpha
\rightarrow J_YT$ pointwise and
$$
\limsup_{\alpha} \|S_\alpha\|_\A \leq \lambda \|J_YT\|_\A \leq \lambda \|T\|_\A.
$$
Denote $E_\alpha=S_\alpha(X) \subset Y^{**}$.

Let us consider the set of triples $(\alpha, F,\ep)$, where $\alpha$ is as
above, $\ep>0$ and $F$ runs over the finite-dimensional subspaces of $Y^*$,
ordered in a natural way. For each $(\alpha, F,\ep)$, using the principle of
local reflexivity, we may find an operator $R_{(\alpha, F,\ep)} \in \mathcal
L(E_{\alpha},Y)$ with $\|R_{(\alpha, F,\ep)}\|\leq 1+\ep$ such that
$$
y^*(R_{(\alpha, F,\ep)}y^{**})=y^{**}(y^*),\quad y^*\in F, y^{**} \in E_\alpha.
$$
Denoting by $\widetilde S_{\alpha}$ the operator $S_\alpha$ considered with
values in
$E_\alpha$, we have (see for instance \cite[Proposition 8.4.4]{Pie}) $\widetilde
S_\alpha \in \A^{inj}(X;E_\alpha)=\A(X;E_\alpha)$ and
$$
\|\widetilde S_\alpha\|_\A=\|S_{\alpha}\|_\A.
$$

Put $T_{(\alpha,F,\ep)}=R_{(\alpha,F,\ep)}\widetilde S_{\alpha}$. Then, the net
$(T_{(\alpha,F,\ep)})$ is in $\F(X;Y)$ and
$$
\|T_{(\alpha,F,\ep)}\|_\A\leq (1+\ep)\|\widetilde S_{\alpha}\|_\A
=(1+\ep)\|S_\alpha\|_\A.
$$
Therefore,
$$
\limsup_{(\alpha,F,\ep)}\|T_{(\alpha,F,\ep)}\|_\A \leq
\limsup_{\alpha}\|S_\alpha\|_\A \leq \lambda \|T\|_\A.
$$
Moreover, if $x \in X$ and $y^* \in Y^*$, we have with $F\subset Y^*$ such that
$y^*\in F$,
$$
y^*(T_{(\alpha,F,\ep)}x)=y^*(R_{(\alpha,F,\ep)} S_{\alpha}x)=(S_\alpha x)(y^*).
$$
Since $(S_\alpha x)(y^*)\rightarrow (J_Y Tx)(y^*)=y^*(Tx)$, we get that
$T_{(\alpha,F,\ep)}\rightarrow T$ in $\tau_w$. After passing to convex
combinations if necessary, we may assume that $T_{(\alpha,F,\ep)}\rightarrow T$
pointwise. Thus, the proof is complete.
\end{proof}

\begin{corollary}\label{Bidual wBAP}
Let $\A$ be an injective Banach operator ideal and $X$ be a Banach space. If
$\overline{\F(X;Y^{**})}^{\|\cdot\|_\A} =\A(X;Y^{**})$ for every Banach space
$Y$, then $X$ has the weak MAP for $\A$.
\end{corollary}

\begin{proof}
Take $T \in \A(X;Y^{**})$. As in the proof of Proposition~\ref{No Prop 5.5},
there is a sequence $(T_n) \subset \F(X;Y^{**})$ such that $T_n \rightarrow T$
pointwise and $\lim_{n} \|T_n\|_\A= \|T\|_\A$. Since $\A$ is injective, by
Proposition~\ref{c* for bidual}, $X$ has the local MAP for $\A$ which, by
Theorem~\ref{wBAP coincide localBAP}, is equivalent to the weak MAP for $\A$.
\end{proof}

Corollary~\ref{Bidual wBAP} will enable us to relate the weak $\lambda$-BAP
for $\A$ with the $\A$-approximation property showing a lifting result (see
Proposition~\ref{A-AP and wBAP} below). Let $\A$ be a Banach operator ideal. As
in \cite{Oja_JMAA} (see also \cite[Definition~4.3]{LaTur}), we say that a
Banach space $X$ has the $\A$-\textit{approximation property} ($\A$-AP) if
$\overline{\mathcal F(Y;X)}^{\|.\|_{\A}} = \A(Y;X)$ for every Banach space $Y$.

Thanks to Grothendieck's classics, the $\K$-AP coincides with the classical
AP. Since $\K^{dual}=\K$, the result below just extends a well-known
Grothendieck's characterization of the AP of dual spaces.

\begin{proposition}\label{prop:A-AP and dens}
Let $\A$ be a Banach operator ideal such that $\A \subset \A^{dual\ dual}$ and
let $X$ be a Banach space. Then $X^{*}$ has the $\A$-AP if and only if
$\overline{\F(X;Y^*)}^{\|\cdot\|_{\A^{dual}}}=\A^{dual}(X;Y^*)$ for every
Banach space $Y$.
\end{proposition}

Proposition~\ref{prop:A-AP and dens} is immediate from the lemma below.

\begin{lemma}\label{2.1.10 phd}
Let $\A$ be a Banach operator ideal such that $\A \subset \A^{dual\ dual}$ and
let $X$ and $Y$ be Banach spaces. Then
$\overline{\F(Y;X^*)}^{\|\cdot\|_\A}=\A(Y;X^{*})$ if and only if
$\overline{\F(X;Y^*)}^{\|\cdot\|_{\A^{dual}}}=\A^{dual}(X;Y^*)$.
\end{lemma}

\begin{proof}
We only show the `if' part, the other one being analogous. Fix $T\in \A(Y;X^{*})$ (hence in $\A^{dual\ dual}(Y;X^{*})$) and $\ep>0$. Since $T^*\in \A^{dual}(X^{**};Y^*)$, $T^*J_X \in \A^{dual}(X;Y^*)$. Take $S \in \F(X;Y^*)$ such that $\|T^*J_X - S\|_{\A^{dual}}=\|J^*_X T^{**} - S^*\|_\A < \ep$. As $T=J^*_XT^{**}J_Y$, we have
$$
\|T- S^*J_Y\|_\A\leq \|J^*_X T^{**} - S^*\|_\A < \ep,
$$
which concludes the proof because $S^*J_Y \in \F(Y;X^{*})$.
\end{proof}

In the next two results, we shall use that $\A^{dual}$ is injective whenever $\A$ is surjective
(see, for instance, \cite[Proposition 8.5.10 (2)]{Pie}).

\begin{proposition}\label{A-AP and wBAP}
Let $\A$ be a surjective Banach operator ideal such that $\A \subset \A^{dual\ dual}$ and let $X$ be a Banach space. If $X^{*}$ has the $\A$-AP, then $X$ has the weak MAP for $\A^{dual}$.
\end{proposition}

\begin{proof}
By Proposition~\ref{prop:A-AP and dens}, if $X^{*}$ has the $\A$-AP, then
$\overline{\F(X;Y^{**})}^{\|\cdot\|_{\A^{dual}}}=\A^{dual}(X;Y^{**})$ for every
Banach space $Y$. Since $\A^{dual}$ is injective, an immediate application of
Corollary~\ref{Bidual wBAP} gives the result.
\end{proof}

If $\A$ is a closed operator ideal, then clearly also $\A^{dual}$ is. Hence, from Proposition~\ref{A-AP and wBAP} and Corollary~\ref{BAP for A = weak BAP for A}, we get the following.

\begin{proposition}\label{A-AP and MAP surjective closed}
Let $\A$ be a surjective closed operator ideal such that $\A \subset \A^{dual\ dual}$ and let $X$ be a Banach space. If $X^{*}$ has the $\A$-AP, then $X$ has the MAP for $\A^{dual}$.
\end{proposition}

In the special case $\A=\K$, recalling that the MAP for $\K$ coincides with the weak MAP (see \cite[Theorem~2.4]{LO}), Proposition~\ref{A-AP and MAP surjective closed} yields an alternative proof of the Lima--Oja result mentioned in the beginning of this section.

A particular case of the $\A$-AP is the $\K_\A$-AP studied in detail in \cite{LaTur2}. Here $\K_\A$ denotes the ideal of \textit{$\A$-compact operators} of Carl and Stephani \cite{CaSt}, those which send bounded sets
into $\A$-compact sets. (Recall that a subset $K$ of $X$ is
$\A$-\textit{compact} if it is contained in the closed absolutely convex hull of an $\A$-null sequence \cite[Theorem~1.1]{CaSt}.) In \cite{LaTur2} $\K_\A$ was equipped with a natural Banach operator ideal norm. Since $\K_\A$ is surjective (see
\cite[Theorem~2.1]{CaSt} and \cite[Proposition~2.1]{LaTur2}) and $\K_\A=\K_\A^{dual\ dual}$ \cite[Corollary 2.4]{LaTur2}, Proposition~\ref{A-AP and wBAP} implies the following.

\begin{corollary}\label{Dual con K_A}
Let $\A$ be a Banach operator ideal and $X$ be a Banach space. If $X^{*}$ has
the $\K_\A$-AP, then $X$ has the weak MAP for $\K_\A^{dual}$.
\end{corollary}

A well-known special case of $\K_\A$ is the Banach operator ideal $\K_p$ of $p$-\textit{compact operators}. This is the case when $\A=\N^p$, the ideal of right $p$-nuclear operators (see \cite[Remark~1.3]{LaTur2}). The $\K_p$-AP was launched by Delgado, Pi\~neiro and Serrano \cite{DPS_dens} under the name of $\kappa_p$-AP. Since $\K=\K_\infty$, the $\K_\infty$-AP coincides with the classical AP. Also, for closed subspaces of an $L_p(\mu)$-space, where $1\le p <\infty$, the $\K_p$-AP is the same as the AP \cite[Theorem~1]{Oja_JMAA}.

By~\cite[Remark 4.3]{AiLiOj} or \cite[Theorem~2.8]{GaLaTur}, $\K_p^{dual} = \mathcal{QN}_p$ isometrically (see also \cite{DPS_adj, Pie2}), where $\mathcal{QN}_p$ is the ideal of quasi $p$-nuclear operators. It is well known that $\mathcal{QN}_p \subset \Ps_p$ isometrically. This leads us to the following lifting result.

\begin{proposition}\label{QNp=Pip}
Let $X$ be a Banach space and let $1\leq p<\infty$. Suppose that $\mathcal{QN}_p(X;Y^{**})=\Ps_p(X;Y^{**})$ for every Banach space $Y$. If $X^{*}$ has the $\K_p$-AP, then $X$ has the weak MAP for $\Ps_p$.
\end{proposition}

\begin{proof}
Suppose that $X^{*}$ has $\K_p$-AP. Since $\K_p^{dual}=\mathcal {QN}_p$ isometrically, by Corollary~\ref{Dual con K_A}, $X$ has the weak MAP for $\mathcal{QN}_p$. Since $\mathcal{QN}_p$ and $\Ps_p$ are regular, a direct application of Corollary~\ref{regular and biduals} completes the proof.
\end{proof}

It is known that $\mathcal{QN}_p(X;Y)=\Ps_p(X;Y)$, $1\le p < \infty$, for all Banach spaces $Y$ whenever $X$ is an Asplund space (equivalently, $X^*$ has the Radon--Nikod\'ym property).
(This result is essentially due to Persson \cite{Per}: his proof for the special case
when $X^*$ is separable or reflexive goes through in the general case; this was firstly
noticed probably in \cite{Rei3} and \cite{Hei}.) Therefore, a direct application
of Proposition~\ref{QNp=Pip} gives the following.

\begin{corollary}\label{Asplund}
Let $X$ be an Asplund Banach space and let $1\leq p<\infty$. If $X^{*}$ has the $\K_p$-AP, then $X$ has the weak MAP for $\Ps_p$.
\end{corollary}

Since $\Ps_p$ is an injective Banach operator ideal, Theorem~\ref{wBAP coincide localBAP} allows us to consider indistinctly the local MAP for $\Ps_p$ or the weak MAP for $\Ps_p$, the latter being, by Theorem~\ref{Saphar_orig}, equivalent to the Saphar MAP of order $p^*$.
It is known that if $X^{**}$ has the BAP of order $p$, then $X$ has it (see for instance \cite[Proposition~21.7]{DF}). Let us discuss how $X^{*}$ is positioned in this framework. A first result of this type can be found in \cite[Corollary~2.9]{DOPS}. Also, relying on \cite[Theorem 4]{Sap}, Delgado, Pi\~neiro and Serrano related the AP of order $p$ with the $\K_p$-AP \cite[Corollary~2.5]{DPS_dens}. (Recall that a Banach space $X$ has the \textit{AP of order $p$}, $1\leq p <\infty$, if for all Banach spaces $Y$, the natural map from $Y^*\widehat{\otimes}_{g_p}X$ (the completion of $Y^*{\otimes}X$ with the Chevet--Saphar tensor norm $g_p$) to $\mathcal L(Y;X)$ is injective.)

\begin{proposition}[Saphar--Delgado--Pi\~neiro--Serrano]\label{X^{**} g_p}
Let $X$ be a Banach space and let $1<p<\infty$. If $X^{**}$ the AP of order $p^*$, then $X^{*}$ has the $\K_{p}$-AP.
\end{proposition}

We do not know if the $\K_p$-AP on $X^{*}$ implies the AP of order $p^*$ on $X$. However, thanks to Corollary~\ref{Asplund} and the above discussion we have the following.

\begin{corollary}\label{Asplund_gp}
Let $X$ be an Asplund Banach space and let $1\leq p<\infty$. If $X^{*}$ has the $\K_p$-AP, then $X$ has the MAP of order $p^*$.
\end{corollary}

For any reflexive space $X$ and any $1<p<\infty$, the AP of order $p$ and the MAP of order $p$
coincide \cite[Theorem~4.2]{Rei}. Since reflexive spaces are Asplund, we have the following.

\begin{corollary}\label{Omnibus Corollary}
Let $X$ be a reflexive Banach space and let $1<p<\infty$. The following statements are equivalent.
\begin{enumerate}[\upshape (i)]
\item $X^{*}$ has the $\K_p$-AP.
\item $X$ has the weak MAP for $\Ps_p$.
\item $X$ has the local MAP for $\mathcal{QN}_p$.
\item $X$ has the MAP of order $p^*$.
\item $X$ has the AP of order $p^*$.
\end{enumerate}
\end{corollary}

The equivalence between (i) and (v) was previously obtained in \cite[Corollary 8]{Oja_JMAA} for $X$ being a quotient of an $L_p(\mu)$-space, $1< p <\infty$.

\section{Relations with approximation properties given by tensor norms}
\label{Sec: tensor norms}

In this section we relate the properties under study with approximation
properties given by tensor norms. In order to proceed we recall some definitions and basic
results. First of all, when we use ``tensor norm'' we follow the terminology of Ryan's book \cite[p. 130]{Ryan} (according to the Defant--Floret \cite{DF} terminology, this is a ``finitely generated tensor norm'').

Let $X$ and $Y$ be Banach spaces and $\alpha$ be a tensor norm. Since $X^*\otimes Y=\F(X;Y)$, we may (and shall) consider the normed space $(\F(X;Y),\alpha(.))$. There is a bijective correspondence between the classes of all maximal Banach operator ideals $\A$ and of all tensor norms $\alpha$, in this case $\A$ and $\alpha$ are said to be \textit{associated} \cite[17.3]{DF}. If $\A$ and $\alpha$ are associated, then for all $X$ and $Y$
\begin{equation}\label{normas tensor vs op}
\|S\|_\A\le \alpha(S), \quad S\in \F(X;Y)
\end{equation}
\cite[17.6]{DF}. In this case also $(X\otimes_{\alpha'} Y)^* =\A(X;Y^*)$ isometrically, where ${\alpha}'$ stands for the dual tensor norm of ${\alpha}$, under the duality
$$
\langle T, u\rangle= \sum_{j=1}^n (Tx_j)(y_j), \quad T\in \A(X;Y^*), \quad u=\sum_{j=1}^n x_j \otimes y_j \in X\otimes Y,
$$
and, similarly, $\A(X;Y)=(X\otimes_{\alpha'} Y^*)^* \cap \mL(X;Y)$ \cite[17.5]{DF}.

Let $\alpha$ be a tensor norm. Recall (from \cite[21.7]{DF}) that a Banach space $X$ has the \textit{bounded $\alpha$-approximation property with constant $\lambda$} ($\alpha$-$\lambda$-BAP) if for every Banach space $Y$ the natural mapping $\jmath \colon
Y^*\otimes_{{\alpha}} X \rightarrow (Y\otimes_{{\alpha}'}X^{*})^*$ satisfies
${\alpha}(u) \leq \lambda \|\jmath(u)\|$, $u\in Y^*\otimes X$. Summarizing, we may clearly reformulate the $\alpha$-$\lambda$-BAP in the form of an `outer inequality' (cf. \cite[Definition~1.3]{Oja_TAMS}) as follows.

Let $1\le\lambda<\infty$. Let a tensor norm $\alpha$ be associated with a maximal Banach operator ideal $\A$. A Banach space $X$ has the $\alpha$-$\lambda$-BAP if and
only if for every Banach space $Y$ and every $S \in \mathcal F(Y; X)$
$$
{\alpha}(S) \leq \lambda \|S\|_{\A(Y;X)}.
$$
Note that the Saphar $\lambda$-BAP of order $p$ is precisely the $g_p$-$\lambda$-BAP \cite[21.7]{DF}.

Let $\A$ be a maximal Banach operator ideal $\A$ and $\alpha$ be associated with $\A$. As usual (see \cite[17.9]{DF} or \cite[p.~197]{Ryan}), we denote by $\A^*$ the \textit{adjoint} Banach operator ideal of $\A$. It is known, that $\A^*$ is maximal and $\A^*$ is associated with the tensor norm $\alpha^*:= (\alpha')^t=(\alpha^t)'$, where $\alpha^t$ denotes the transpose of $\alpha$. The following result connects the $\alpha$-BAP with the local BAP for $\A^*$. It also provides a lifting result for the $\alpha$-BAPs from dual spaces down to underlying spaces.

\begin{theorem}\label{Teo Saphar Gral}
Let $\A$ be a maximal Banach operator ideal associated with a tensor norm $\alpha$. Let $1\le\lambda, \tilde\lambda<\infty$ and $X$ be a Banach space. Then the following statements hold.
\begin{enumerate}[\upshape (a)]
\item If $X$ has the ${\alpha}$-$\lambda$-BAP, then $X$ has the local $\lambda$-BAP
for $\A^*$.
\item If $X$ has the local $\lambda$-BAP for $\A^*$ and $X^{*}$ has the
${\alpha}'$-$\tilde\lambda$-BAP, then $X$ has
the ${\alpha}$-$\lambda\tilde\lambda$-BAP.
\end{enumerate}
\end{theorem}
\begin{proof}
Suppose that $X$ has the $\alpha$-$\lambda$-BAP and take $T\in \A^*(X;Y)$.
Since $\A^*$ is maximal, it is regular \cite[Corollary 17.8.2]{DF} and then $J_Y
 T \in \A^*(X;Y^{**})$ with $\|J_Y T\|_{\A^*}=\| T\|_{\A^*}$. Now, by
\cite[Proposition 21.8]{DF}, which describes the $\alpha$-BAP of $X$ as a property of an `approximation' of operators from $\A^*(X;Y^{**})$, there exists a net $(S_{\nu})$ in $\F(X;Y)$ such that $S_{\nu}\rightarrow T$ in the weak operator topology on $\mL(X;Y)$ and
$\sup_{\nu} {\alpha}^*(S_{\nu})\leq \lambda \| J_YT\|_{\A^*}=\lambda \|T\|_{\A^*}$. After passing to convex combinations, we may assume that $S_{\nu}\rightarrow T$ pointwise and by \eqref{normas tensor vs op} we have $\limsup_{\nu} \|S_{\nu}\|_{\A^*}\leq \lambda \| T\|_{\A^*}$. Hence, $X$ has the local $\lambda$-BAP for $\A^*$.

To prove the second statement, take $S \in \F(Y;X)$, $S=\sum_{j=1}^n y_j^*\otimes x_j$ with $\sum_{j=1}^n \|y^*_j\|=1$. Since
$$
\alpha(S) = \alpha^{**}(S)=((\alpha^{*})')^t(S)=(\alpha^{*})'\big(\sum_{j=1}^n x_j\otimes y_j^*\big)
$$
and
$(X\otimes_{({\alpha}^*)'}Y^*)^*=\A^{*}(X;Y^{**})$, there is $T\in
\A^*(X;Y^{**})$ with $\|T\|_{\A^*}=1$ such that
\begin{equation}\label{alphanorm}
{\alpha}(S)=|\sum_{j=1}^n (Tx_j)(y_j^*)|.
\end{equation}
Since $X$ has the local $\lambda$-BAP for $\A^*$, given $\ep>0$ there is
$T_0\in \mathcal F(X;Y^{**})$ such that $\|T_0\|_{\A^*}\leq \lambda$ and
$\|Tx_j-T_0x_j\|\leq \ep$, $j=1,\ldots,n$. Then, from \eqref{alphanorm} we get
$$
{\alpha}(S)\leq|\sum_{j=1}^n (T_0x_j)(y^*_j)|+\ep.
$$

Let us consider $T_0^*J_{Y^*} \in \F(Y^*;X^*)$ as an element of $Y^{**}\otimes_{\alpha'} X^*$. As was mentioned before, $(Y^{**}\otimes_{\alpha'} X^*)^*=\A(Y^{**};X^{**})$ isometrically. Since $S^{**}\in \F(Y^{**};X^{**})\subset \A(Y^{**};X^{**})$ we may write
$$
\Big|\sum_{j=1}^n (T_0x_j)(y^*_j)\Big| =\big|\langle S^{**}, T_0^*J_{Y^*}\rangle\big| \le \|S^{**}\|_\A\alpha'(T_0^*J_{Y^*}).
$$
Now, as $X^{*}$ has the ${\alpha}'$-$\tilde\lambda$-BAP and $(\A^*)^{dual}$ is
the maximal Banach operator ideal associated with ${\alpha}'$,
$$
\alpha'(T_0^*J_{Y^*})\le \tilde \lambda \|T_0^*J_{Y^*}\|_{(\A^*)^{dual}} \le \tilde \lambda \|T_0^*\|_{(\A^*)^{dual}}\le \tilde \lambda \|T_0^{**}\|_{(\A^*)}.
$$

Finally, recall the fact \cite[Corollary~17.8.4]{DF} that $\B^{dual\ dual}=\B$ isometrically whenever $\B$ is a maximal Banach operator ideal. In our case, this gives that $\|S^{**}\|_\A=\|S\|_\A$ and $\|T_0^{**}\|_{\A^*}=\|T_0\|_{\A^*}\le \lambda$.
Hence,
$$
\alpha(S)\le \tilde \lambda \lambda \|S\|_\A + \ep,
$$
and that is what we need, because $\ep$ is arbitrary.
\end{proof}

There is a class of tensor norms $\alpha$, called {totally accessible}, for which all Banach spaces have the $\alpha$-MAP~\cite[Proposition 21.7]{DF}. Below we shall use the fact that $g_{p}^*$ and $g_{p^*}'=/d_{p}$ (where $/d_{p}$ is the left injective associate of the Chevet--Saphar tensor norm $d_{p}$) are totally accessible~\cite[Corollary 21.1]{DF} and~\cite[Corollary~7.15 and Theorem 7.20]{Ryan}, implying that all Banach spaces enjoy the $g_{p}^*$-MAP and the $g_{p^*}'$-MAP, $1\le p\le \infty$. Recall also that the ideal of $p$-integral operators $\I_p$ is the maximal ideal associated with the tensor norm $g_p$ and that $\mathcal I_p^*=\Ps_{p^*}$, $1\le p\le \infty$, \cite[17.12]{DF}. As easy applications of the above theorem, we first recover the Shaphar characterization of the $\lambda$-BAP of order $p^*$, and then exhibit an example of the local MAP for $\A$ enjoyed by all Banach spaces, where $\A$ is not minimal (compare with Proposition~\ref{No Prop 5.5}).

\begin{corollary}[Saphar, see Theorem~\ref{Saphar_orig}]\label{Saphar again}
Let $1\le \lambda < \infty$ and $1\le p \le \infty$. A Banach space $X$ has the local $\lambda$-BAP for $\Ps_p$ if and only if $X$ has the $\lambda$-BAP of order $p^*$.
\end{corollary}

\begin{proof}
As was mentioned, the $\lambda$-BAP of order $p^*$ is precisely the $g_{p^*}$-$\lambda$-BAP and any Banach space enjoys the $g_{p^*}'$-MAP. Since $\Ps_p=(\I_{p^*})^*$, with $\A=\I_{p^*}$ and $\alpha=g_{p^*}$, Theorem~\ref{Teo Saphar Gral} establishes the equivalence between the both approximation properties.
\end{proof}

Recall from Corollary~\ref{Coro Saphar} that the local $\lambda$-BAP for $\Ps_p$ is equivalent to the weak $\lambda$-BAP for $\Ps_p$.

\begin{corollary}\label{c^* I_p}
Every Banach space has the local MAP for $\I_p$, $1\le p\le \infty$.
\end{corollary}

\begin{proof}
We know that $\I_p^*$ is associated with $g_p^*$, every Banach space has the $g_p^*$-MAP and $\I_p=(\I_p^*)^*$. Hence, by Theorem~\ref{Teo Saphar Gral} (a), every Banach space has the local MAP for $\I_p$.
\end{proof}

Our final aim is to show that, unlike the local BAP for $\I_p$, there exist Banach spaces which fail the weak BAP for $\I_p$. To this end, we need to recall the $p$-approximation property.

Let us denote by $\tau_p$ the topology of uniform convergence on $p$-compact sets. The \textit{$p$-approximation property} ($1\le p <\infty$) of a Banach space $X$ means that the identity map $I_X$ can be approximated in $\tau_p$ by finite rank operators. The class of $p$-compact sets was first introduced and studied in \cite{SiKa} together with the notion of the $p$-approximation property. With the notion of $\A$-compact sets (see Section~\ref{Sec: lifting}), by \cite[Remark~1.3]{LaTur2}, we know that $p$-compact sets coincide with $\mathcal
N^p$-compact sets, where $\N^p$ denotes the ideal of right $p$-nuclear operators. Associated to the class of $p$-compact sets we have the Banach operator ideal $\K_p=\K_{\N^p}$
of $p$-compact operators. For more information on $p$-compact sets and $p$-compact operators we refer the reader to \cite{AiLiOj,CK, DPS_dens, DPS_adj, GaLaTur, Oja_JMAA, Oja_JFA, Pie2} and references therein. Let us remark that the `limit' case $p=\infty$ would just give compact sets, compact operators and the classical AP.

\begin{lemma}\label{Pi-p y right Pi-p equals}
Let $X, Y$ be Banach spaces and $1\le p <\infty$. The following statements hold.
\begin{enumerate}[\upshape (a)]
\item $\mathcal G^{\N^p} = \mathcal G_{\N^p}$ as
linear subspaces of $Y^*\widehat \otimes_\pi X$.
\item $c_{0,\N^p}(X)=c_{0,\N_p^{dual}}(X)=c_{0,{\I}_p^{dual}}(X)= c_{0,\Ps_p^{dual}}(X)$.
\item $\overline{M}^{\tau_p} = \overline {M}^{\tau^\A}$ for any absolutely convex subset $M$ of $\mL(X;Y)$ whenever $\A$ is $\N^p,
\N_p^{dual}, \I_p^{dual}$ or $\Ps_p^{dual}$.
\end{enumerate}
\end{lemma}

\begin{proof}
Statement (a) can be proved following the proof of \cite[Theorem~2.7]{CK}. Let us prove (b). It is well known that $\N_p \subset \I_p\subset \Ps_p$. Hence, $\N_p^{dual} \subset \I_p^{dual} \subset \Ps_p^{dual}$. We also have the inclusion $\N^p \subset \N_p^{dual}$. Indeed, by definition, $\N^p=\N_{(p,1,p)}$ and $\N_p=\N_{(p,p,1)}$, particular cases of general $(u,s,t)$-nuclear operators \cite[18.1.1]{Pie}. Since $\N_{(p,1,p)}^{reg}=\N_{(p,p,1)}^{dual}$ \cite[Theorem~18.1.6]{Pie}, $\N^p \subset \N_p^{dual}$ as claimed. The above inclusions immediately yield that
$$
c_{0,\N^p}(X)\subset c_{0,\N_p^{dual}}(X)\subset c_{0,{\I}_p^{dual}}(X)\subset c_{0,\Ps_p^{dual}}(X).
$$
The missing link $c_{0,\N^p}(X)=c_{0,{\Ps}_p^{dual}}(X)$ is provided by~\cite[Corollary~3.4]{AO2}. Finally, to show (c) we appeal to \cite[p.~73]{DelPin2} implying that the classes of $\tau_p$- and $\tau_{\N^p}$-continuous functionals on $\mathcal L(X;Y)$
coincide. Therefore, the result follows from (a) and (b).
\end{proof}

\begin{proposition}\label{weak BAP and p-AP sika}
Let $\A$ be $\N_p, \I_p$ or $\Ps_p$, $1\le p <\infty$. If a Banach space $X$ has the weak BAP for $\A$, then $X$ has the $p$-approximation property.
\end{proposition}

\begin{proof}
Suppose that $X$ has the weak $\lambda$-BAP for $\A$ for some $\lambda$. By Theorem~\ref{Omnibus Thm}, $X$ has the $\lambda$-BAP for $\A$ and $\tau^{\A^{dual}}$. Then, $I_X\in \overline{\F(X)}^{\tau^{\A^{dual}}}$ (to see this just take $T=0$ in Definition~\ref{def: BAP for A and tau}). By the above lemma, $I_X \in \overline{\F(X)}^{\tau_p}$, which means that $X$ has the $p$-approximation property.
\end{proof}

As a consequence of the above proposition, for $p>2$ the local BAP for $\I_p$ differs from the weak BAP for $\I_p$. The same happens with the ideal $\N_p$.

\begin{proposition}\label{c* and weak BAP Ip}
Let $2<p<\infty$. There is a Banach space which has the local MAP for $\I_p$ and the
local MAP for $\N_p$, but lacks the weak BAP for $\I_p$ and the weak BAP for $\N_p$.
\end{proposition}

\begin{proof}
Fix $2<p<\infty$. Reinov's result \cite[Theorem~5.3,1]{Rei3} clearly implies that there exists
a Banach space $X$ that fails the $p$-approximation property. Now, by Proposition~\ref{weak BAP and p-AP sika}, $X$ fails to have the weak BAP for $\I_p$ and the weak BAP for $\N_p$. On the other hand, by Corollary~\ref{c^* I_p} and Proposition~\ref{No Prop 5.5}, $X$ has the local MAP for $\I_p$ and local MAP for $\N_p$ for all $p$.
\end{proof}

\subsection*{Acknowledgements} We are grateful to Joe Diestel for helpful and valuable conversations. S. Lassalle and P. Turco were supported in part by CONICET PIP 0624 and PIP 0483, PICT 2011-1456, UBACyT 1-746 and 1-474. The research of E. Oja was partially supported by Estonian Science Foundation Grant 8976 and by institutional research funding IUT20-57 of the Estonian Ministry of Education and Research.


\begin{thebibliography}{99}
\bibitem{AiLiOj}
K. Ain, R. Lillemets, E. Oja.
\textit{Compact operators which are defined by $\ell_p$-spaces.}
Quaestiones Math. 35 (2012), 145--159.

\bibitem{AO2}
K. Ain, E. Oja. \textit{On $(p,r)$-null sequences and their relatives.}
Math. Nachr. (2015) doi: 10.1002/mana.201400300.

\bibitem{CaSt}
B. Carl, I. Stephani.
\textit{On $A$-compact operators, generalized entropy numbers and entropy idelas.}
Math. Nachr. 199 (1984), 77--95.

\bibitem{Cas}
P. G. Casazza.
\textit{Approximation properties.}
Handbook of the Geometry of Banach Spaces, Vol. I, 271--316. North-Holland, Amsterdam, 2001.

\bibitem{CK}
Y. S. Choi, J. M. Kim.
\textit{The dual space of $(\mathcal L(X,Y);\tau_p)$ and the p-approximation property.}
J. Funct. Anal. 259 (2010), 2437--2454.

\bibitem{DF}
A. Defant, K. Floret.
Tensor Norms and Operators Ideals.
North Holland Publishing Co., Amsterdam, 1993.

\bibitem{DOPS}
J. M. Delgado, E. Oja, C. Pi\~neiro, E. Serrano.
\textit{The $p$-approximation property in terms of density of finite rank operators.}
J. Math Anal, Appl. 354 (2009), 159--164.

\bibitem{DelPin2}
J.M. Delgado, C. Pi\~neiro.
\textit{An approximation property with respect to an operator ideal.}
Studia Math. 214 (2013), 67--75.

\bibitem{DPS_dens}
J. M. Delgado, C. Pi\~neiro, E. Serrano.
\textit{Density of finite rank operators in the Banach space of $p$-compact operators.}
J. Math. Anal. Appl. 370 (2010), 498--505.

\bibitem{DPS_adj}
J. M. Delgado, C. Pi\~neiro, E. Serrano.
\textit{Operators whose adjoints are quasi $p$-nuclear.}
Studia Math. 197 (2010), 291--304.

\bibitem{dfs}
J. Diestel, J. H. Fourie, J. Swart.
The Metric Theory of Tensor Products, Grothendieck's R{\'e}sum{\'e} Revisited.
American Math. Soc., Providence, RI, 2008.

\bibitem{djt}
J. Diestel, H. Jarchow, A. Tonge.
Absolutely Summing Operators.
Cambridge Studies in Advanced Mathematics, 43. Cambridge University Press, Cambridge, 1995.

\bibitem{DS}
N. Dunford, J. T. Schwartz.
Linear Operators. Part I. General Theory.
Interscience Publishers, New York, 1958.

\bibitem{GaLaTur}
D. Galicer, S. Lassalle, P. Turco.
\textit{The ideal of $p$-compact operators: a tensor product approach.}
Studia Math. 211 (2012), 269--286.

\bibitem{GoGu00}
M. Gonz\'alez, J. Guti\'errez.
\textit{Surjective factorization of holomorphic mappings.}
Comment. Math. Univ. Carolin. 41 (2000), 469--476.

\bibitem{Gro}
A. Grothendieck.
\textit{Produits tensoriels topologiques et espaces nucl\'eaires.}
Mem. Amer. Math. Soc. 16 (1955).

\bibitem{Hei} S. Heinrich.
\textit{Some properties of operators spaces.}
Serdica 3 (1977) 168-175.

\bibitem{LaTur}
S. Lassalle, P. Turco.
\textit{On $p$-compact mappings and the $p$-approximation properties.}
J. Math. Anal. Appl. 389 (2012), 1204--1221.

\bibitem{LaTur2}
S. Lassalle, P. Turco.
\textit{The Banach ideal of $\A$-compact operators and related approximation properties.}
J. Funct. Anal. 265 (2013), 2452--2464.

\bibitem{LNO}
A. Lima, O. Nygaard, E. Oja.
\textit{Isometric factorization of weakly compact operators and the approximation property.}
Israel J. Math. 119 (2000), 325--348.

\bibitem{LLO1}
A. Lima, V. Lima, E. Oja.
\textit{Bounded approximation properties via integral and nuclear operators.}
Proc. Amer. Math. Soc. 138 (2010), 287--297.

\bibitem{LLO2}
A. Lima, V. Lima, E. Oja.
\textit{Bounded approximation properties in terms of $C[0,1]$.}
Math. Scand. 110 (2012), 45--58.

\bibitem{LO}
A. Lima, E. Oja.
\textit{The weak metric approximation property.}
Math. Ann. 333 (2005), 471--484.

\bibitem{LT}
J. Lindenstrauss, L. Tzafriri.
Classical Banach Spaces I.
Vol. 92, Springer-Verlag, Berlin, New York, 1977.

\bibitem{Liss}
A. Lissitsin.
\textit{A unified approach to the strong approximation property and the weak bounded approximation property of Banach spaces.}
Studia Math. 211 (2012), 199--214.

\bibitem{Oja_strong}
E. Oja.
\textit{The strong approximation property.}
J. Math. Anal. Appl. 338 (2008), 407--415.

\bibitem{Oja_Survey}
E. Oja.
\textit{On bounded approximation properties of Banach spaces.}
Banach Algebras 2009, Vol. 91, 219--231. Banach Center Publications, Polish Acad. Sci.
Inst. Math., Warsaw, 2010.

\bibitem{Oja_TAMS}
E. Oja.
\textit{Inner and outer inequalities with applications to approximation properties.}
Trans. Amer. Math. Soc. 363 (2011), 5827--5846.

\bibitem{Oja_CIDAMA}
E. Oja.
\textit{Bounded approximation properties via Banach operator ideals.}
Advanced Courses of Mathematical Analysis IV, 196--215, World Sci. Publ., Hackensack, NJ, 2012.

\bibitem{Oja_JMAA}
E. Oja.
\textit{A remark on the approximation of p-compact operators by finite-rank operators.}
J. Math. Anal. Appl. 387 (2012), 949--952.

\bibitem{Oja_JFA}
E. Oja.
\textit{Grothendieck's nuclear operator theorem revisited with an application to p-null sequences.}
J. Funct. Anal. 263 (2012), 2876--2892.

\bibitem{Per}
A. Persson.
\textit{On some properties of $p$-nuclear and $p$-integral operators.}
Studia Math. 33 (1969) 213--222.

\bibitem{Pie}
A. Pietsch.
Operators Ideals.
Deutsch. Verlag Wiss., Berlin, 1978; North-Holland Publishing Company, Amsterdam, New York, Oxford, 1980.

\bibitem{Pie2}
A. Pietsch.
\textit{The ideal of $p$-compact operators and its maximal hull.}
Proc. Amer. Math. Soc. 142 (2014), 519--530.

\bibitem{Rei2}
O. Reinov.
\textit{Operators of type RN in Banach spaces.} Sov. Math. Dokl. 16
(1975), 119-123; translation from Dokl. Nauk SSSR 220 (1975), 528--531 (Russian).

\bibitem{Rei}
O. Reinov.
\textit{Approximation properties of order $p$ and the existence of non-$p$-nuclear operators with $p$-nuclear second adjoints.}
Math. Nachr. 109 (1982), 125--134.

\bibitem{Rei3}
O. Reinov.
\textit{Disappearance of tensor elements in the scale of $p$-nuclear operators.}
Operator Theory and Function Theory, Vol. 1, Leningrad Univ., Leningrad, 1983, pp. 145--165 (Russian).

\bibitem{Ryan}
R. Ryan.
Introduction to Tensor Products on Banach Spaces.
Springer, London, 2002.

\bibitem{Sap}
P. D. Saphar.
\textit{Hypoth\`ese d'approximation \`a l'ordre $p$ dans les espaces de Banach et approximation d'applications $p$ absolument sommantes.}
Israel J. Math. 13 (1972), 379--399.

\bibitem{SiKa}
D. P. Sinha, A. K. Karn.
\textit{Compact operators whose adjoints factor through subspaces of $\ell_p$.}
Studia Math. 150 (2002), 17--33.

\bibitem{Sza}
A. Szankowski.
\textit{Subspaces without the approximation property.}
Israel J. Math. 30 (1978), 123--129.
\end{thebibliography}
\end{document}